\documentclass[12pt,reqno]{article}
mm \textheight=8.9in

\usepackage{amssymb}
\usepackage{amsmath}
\usepackage{amsthm}
\usepackage{pb-diagram}
\usepackage{color}

\newcommand{\br}[3]{{$#1$}$\lower4pt\hbox{$\tp\atop\raise4pt \hbox{$\scriptscriptstyle{#2}$}$} ${$#3$}}
\newcommand{\tw}[3]{{$#1$}${\,\scriptscriptstyle {#2}}\atop\raise9pt\hbox{$\scriptstyle\tp$} ${$#3$}}
\newcommand{\ttps}[2]{{#1}\raise5pt\hbox{$\lower12pt\hbox{$\scriptstyle\tp$}\atop \lower0pt\hbox{$\tilde\;$}$}\raise4.5pt\hbox{${\scriptstyle{#2}}$}}
\newcommand{\st}[1]{\mbox{${\,\scriptscriptstyle {#1}}\atop\raise5.5pt\hbox{$*$}$}}

\newcommand{\rd}[1]{\mbox{${\,\scriptscriptstyle {#1}}\atop\raise5.5pt\hbox{$\bullet$}$}}
\newcommand{\rt}[1]{\otimes_\chi}
\newcommand{\lt}[1]{\mbox{${\,\scriptscriptstyle {#1}}\atop\raise5.5pt\hbox{$\ltimes$}$}}
\newcommand{\btr}{\raise1.2pt\hbox{$\scriptstyle\blacktriangleright$}\hspace{2pt}}
\newcommand{\btl}{\raise1.2pt\hbox{$\scriptstyle\blacktriangleleft$}\hspace{2pt}}

\newcommand{\lcr}{\raise1.0pt \hbox{${\scriptstyle\rightharpoonup}$}}
\newcommand{\rcr}{\raise1.0pt \hbox{${\scriptstyle\leftharpoonup}$}}

\newcommand{\ttp}{{\lower12pt\hbox{$\tp$}\atop \hbox{$\tilde\;$}}}

\newcommand{\id}{\mathrm{id}}

\newcommand{\Fc}{\mathcal{F}}
\newcommand{\Hc}{\mathcal{H}}

\newcommand{\Jc}{\mathcal{J}}

\renewcommand{\S}{\mathcal{S}}
\newcommand{\Ru}{\mathcal{R}}

\newcommand{\C}{\mathbb{C}}
\newcommand{\Z}{\mathbb{Z}}

\newcommand{\N}{\mathbb{N}}

\newcommand{\tp}{\otimes}

\newcommand{\U}{U}

\newcommand{\ve}{\varepsilon}
\newcommand{\gm}{\gamma}
\newcommand{\dt}{\delta}

\newcommand{\ot}{\otimes}
\newcommand{\op}{\oplus}
\newcommand{\la}{\lambda}

\newcommand{\End}{\mathrm{End}}

\newcommand{\rk}{\mathrm{rk}}

\newcommand{\Tr}{\mathrm{Tr}}

\newcommand{\Rm}{\mathrm{R}}

\newcommand{\La}{\Lambda}
\newcommand{\Gm}{\Gamma}
\newcommand{\g}{\mathfrak{g}}
\renewcommand{\b}{\mathfrak{b}}

\newcommand{\h}{\mathfrak{h}}

\newcommand{\s}{\mathfrak{s}}

\newcommand{\eps}{\epsilon}

\newcommand{\nn}{\nonumber}

\renewcommand{\l}{\mathfrak{l}}

\newcommand{\si}{\sigma}
\newcommand{\al}{\alpha}

\newcommand{\bt}{\beta}

\newcommand{\prt}{\partial}

\newcommand{\be}{\begin{eqnarray}}
\newcommand{\ee}{\end{eqnarray}}

\newtheorem{thm}{Theorem}[section]
\newtheorem{propn}[thm]{Proposition}
\newtheorem{lemma}[thm]{Lemma}
\newtheorem{corollary}[thm]{Corollary}

\newtheorem{remark}[thm]{Remark}

\newcount\prg

\newcommand{\parag}{\advance\prg by1 {\noindent\bf\thesection.\the\prg\hspace{6pt}}}

\begin{document}
\title{ R-matrix and inverse Shapovalov form}
\author{
Andrey Mudrov\vspace{20pt}\\
\small Department of Mathematics,\\ \small University of Leicester, \\
\small University Road,
LE1 7RH Leicester, UK\\
}

\date{}
\maketitle

\begin{abstract}
We  construct the inverse Shapovalov form of a simple complex quantum group from its universal R-matrix
based on a generalized Nagel-Moshinsky approach to lowering operators.
We establish a connection between this algorithm and the ABRR equation for  dynamical twist.
\end{abstract}


\section{Introduction}
There are two important bilinear forms related to the triangular decomposition $U_q(\g)=U_q(\g_-)\tp U_q(\h)\tp U_q(\g_+)$ of the Drinfeld-Jimbo quantum group $U_q(\g)$.
One of them has a classical analog that plays a role in the repsentation theory of simple
Lie algebras. This form was introduced by Shapovalov and studied by Jantzen for, respectively, ordinary and parabolic Verma modules, \cite{Shap,Jan}. It is responsible for irreducibility  of highest weight modules while its inverse, for generic weights,
yields singular vectors in tensor product  modules and invariant star product on coadjoint
orbits, \cite{AL}. This is also applicable to quantized universal enveloping algebras, \cite{D}, whose representation theory is in parallel to the classical setting, and
the quantum Shapovalov form can be as well defined on quantum Verma modules and their
quotients, \cite{Jan,DCK}.

At the same time, there is another impotant object that is special to quantum groups and that can be interpreted as a bilinear pairing between the Borel subalgebras $U_q(\b_\pm)=U_q(\h)U_q(\g_\pm)$. They are in duality as Hopf algebras, and the inverse form is given by the universal R-matrix $\Ru\in U_q(\b_+)\tp U_q(\b_-)$, \cite{D}.
We establish a relation between the two forms  constructing the Shapovalov inverse from $\Ru$.
As the latter is explicitly expressed through generators of the quantum group, \cite{LS, KT}, we result in an explicit
formula for the inverse Shapovalov form.

Our approach originates from the  Nagel-Moshinsky
construction of the lowering/raising operators of the classical non-exceptional Lie algebras, \cite{NM,Mol}. In combination with the R-matrix, it was applied to the natural representations  of
orthosymplectic quantum groups
in \cite{AM}, to construct generators of Mickelsson algebras. In this paper we extend those ideas to all  Drinfeld-Jimbo quantum groups and all weight modules, $V$.
As a result, we obtain explicit expressions for singular vectors in tensor products $V\tp M_\la$, where
 $M_\la$ is a Verma module of generic highest weight $\la$. For $V$ a fundamental representation, that yields lowering operators
of Mickelsson algebras. For $V$ the negative Verma module $M_\la^*$ of lowest weight $-\la$, the singular vector of zero weight in $M_\la^*\tp M_\la$ is
the inverse of the invariant Shapovalov form $M_\la \tp M^*_\la\to \C$.

We also construct a lift of the Shapovalov inverse to $\hat  U_q(\b_+)\tp \hat U_q(\b_-)$ (the hat means extension over a ring
of fractions of $U_q(\h)$) in purely algebraic terms, as a "universal tensor".
The key element of the theory is tensor $\Fc=\frac{1}{q-q^{-1}}(q^{-h\dot \tp h}\Ru-1\tp 1)\in U_q(\g_-)\tp U_q(\g_+)$, where
$h\dot \tp h\in \h\tp \h$ is the inverse of the canonical inner form on $\h^*$. It  satisfies
 a sort of intertwining identity that is equivalent to $\Ru\Delta(x)=\Delta^{op}(x)\Ru$, $x\in U_q(\g)$.
In the classical limit, $\Fc$ tends to $\sum_{\al \in \Rm^+}e_\al\tp f_\al\in \g_+\tp \g_-$ (the inverse of the ad-invariant
form restricted to $\g_-\tp \g_+$) as $\hbar =\log q$ goes to $0$,
so the theory applies to the classical universal enveloping algebras either as a limit case or directly, despite of degeneration of the R-matrix.

Our main result can be described as follows.
Choosing an orthogonal basis $\{h_s\}\subset \h$
we set $h\cdot h=\sum_{s=1}^{\rk \g}h_s^2\in U_\hbar(\h)$.
Let  $h_\rho\in \h$ be the Weyl vector, satisfying $\al(h_\rho)=\frac{1}{2}(\al,\al)$
for all simple positive roots $\al$, and define $d=h_\rho + \frac{h\cdot h}{2}\in U_\hbar(\h)\tp U_\hbar(\h)$.
Passing to $\hbar$-adic completion $U_\hbar(\b_\pm)$,
introduce  a linear  operator $D$ on the zero weight subalgebra in $U_\hbar(\g_+)\g_+\tp U_\hbar(\b_-)\g_-$ via the commutator
$D(X)=[1\tp d,X]$.
The operator $\frac{q^{2D}-\id}{q-q^{-1}}$ preserves the zero weight subalgebra in
$U_q(\g_+)\g_+\tp \hat U_q(\b_-)\g_-$ and is invertible on it. Here the hat designates extension over a localized
Cartan subalgebra.
Put $\Fc^{(0)}= 1\tp 1$ and, by induction,  $\Fc^{(k+1)}=(q-q^{-1})(q^{-2D}-\id)^{-1}\bigl( \Fc\Fc^{(k)}\bigr)\in U_q(\g_+)\tp \hat U_q(\b_-)$ for $k\geqslant 0$. Then the series $\hat \Fc=\sum_{k=0}^\infty\Fc^{(k)}$ is
 a lift of the inverse of the Shapovalov form to $U_q(\g_+)\tp \hat U_q(\b_-)\subset \hat  U_q(\b_+)\tp \hat U_q(\b_-)$.
Remark that the Cartan subalgebra is sitting in the right tensor leg of $\hat \Fc$, which is special to the
 method.

It follows that  $\hat \Fc$ satisfies a linear equation, which is close to the ABRR equation
for  dynamical twist, \cite{ES,ABRR}. Thus we explain the ABRR equation from the generalized Nagel-Moshinsky algorithm.
In contrast with the solution given in  \cite{ABRR} in the form of infinite product, we present it
as a series truncate in every finite dimensional representation while
the product is still infinite.

The paper is organized as follows.
Section 2 collects the basics on  quantum groups and R-matrix.
There we define the matrix $\Fc$. In Section 3, we introduce partial ordering on weight
lattices of representations and define a certain map $\h^*\to \h+\C$ that plays a role in the sequel.
Section 4 is devoted to "localization"
of the action of positive Chevalley generators on a Verma module. With this data
we construct a matrix $\hat F$, which is the image
of $\hat \Fc$ in $\End(V)\tp \hat U_q(\b_-)$. This part of the paper
is a direct generalization of \cite{NM}.
In Section 5 we construct the universal matrix $\Fc$ which is specializing
to $\hat F$ in the representation $V$. We prove that  $\hat \Fc$
is the inverse Shapovalov form in Section 6. As an intermediate step, we construct
singular vectors in the tensor product of $V$ with a Verma module.
Finally, we demonstrate the relation of our construction with the ABRR approach.

\section{Preliminaries}
Let $\g$ be a simple complex Lie algebra with a fixed Cartan subalgebra $\h$ and triangular decomposition
$\g=\g_-\oplus \h\oplus \g_+$. Denote by $\Rm$ the root system of $\g$ with the subsystem of
positive roots $\Rm^+\subset \Rm$ and
the basis of simple roots  $\Pi^+\subset \Rm^+$.
Let $(\hspace{2pt}\cdot\>, \cdot\hspace{2.5pt})$ designate the canonical inner product on $\h^*$.
We normalize it so that $(\al,\al)\in \N$ for all $\al\in \Pi^+$.
Assuming $q\in\C$ is not a root of unity, we write $[z]_q=\frac{q^z-q^{-z}}{q-q^{-1}}$ whenever $q^{\pm z}$ make sense. The $q$-factorial is defined as $[n]_q!=[1]_q\cdot [2]_q\cdot\ldots \cdot[n]_q$ for $n\in \Z$.

Put $q_\al=q^{\frac{(\al,\al)}{2}}$. The quantum group $U_q(\g)$ is a $\C$-algebra generated by $e_{\pm\al}$ and $q^{\pm h_\al}$ subject to
\be
q^{h_\al}e_{\pm\bt}q^{-h_\al}=q^{\pm (\al,\bt)} e_{\pm\al},\quad [e_\al,e_{-\bt}]=\dt_{\al\bt}[h_\al]_q,
\label{Cart_rel}
\ee
$q^{h_\al}q^{-h_\al}=1=q^{-h_\al}q^{h_\al}$, plus the quantized Serre relations
$$
\sum_{k=0}^{1-a_{\al \bt}}(-1)^k
\frac{[n]_{q_{\al}}!}{[k]_{q_\al}![n-k]_{q_\al}!}
e_{\pm \al}^{1-a_{\al \bt}-k}
e_{\pm \bt}e_{\pm \al}^{k}
=0
, \quad \forall \al, \bt \in \Pi^+,
$$
where $a_{\al \bt}=\frac{2(\al,\bt)}{(\al,\al)}$ are the entries of the Cartan matrix.
The assignment $e_{\pm \al}\mapsto e_{\mp \al}$, $q^{\pm h_\al}\mapsto q^{\pm h_\al}$
extends to an antialgebra automorphism $\si$ called Chevalley involution.

 Note that the right identity (\ref{Cart_rel}) does not involve
$q_\al$ as in the usual definition of $U_q(\g)$, see e.q. \cite{CP}. It is a result of  rescaling of
the conventional negative Chevalley generators, $e_{-\bt}\mapsto [\frac{(\bt,\bt)}{2}]_q e_{-\bt}$, which does not affect
the homogeneous Serre relations.

We work with the comultiplication defined on the generators by
\be
&\Delta(e_{\al})=e_{\al}\tp q^{h_{\al}} + 1\tp e_{\al},
\quad
\Delta(e_{-\al})=e_{-\al}\tp 1 + q^{-h_{\al}} \tp e_{-\al},
\nn\\&
\Delta(q^{\pm h_{\al}})=q^{\pm h_{\al}}\tp q^{\pm h_{\al}},
\nn
\ee
for all $\al \in \Pi^+$.
The counit acts by $\eps(1)=1=\eps(q^{\pm h_\al})$ and $\eps(e_{\pm\al})=0$.
 The antipode is determined by $\gm(q^{\pm h_{\al}})=q^{\mp h_{\al}}$, $\gm(e_\al)=-e_\al q^{-h_{\al}}$,
$\gm(e_{-\al})=-q^{h_{\al}}e_{-\al}$.

For all $\bt \in \Rm$ there are elements $e_\bt \in U_q(\g)$ (higher root vectors), which generate a Poincare-Birghoff-Witt type basis in
$U_q(\g)$ over $U_q(\h)$, \cite{CP}.
We denote by $U_q(\g_\pm)$ the subalgebras in $U_q(\g)$ generated by, respectively, $\g_\pm=\{e_{\pm\al}\}_{\al\in \Rm^+}$.
Furthermore, the
$U_q(\b_\pm)\subset U_q(\g)$ denote the quantum Borel subgroups generated by $U_q(\g_\pm)$ and  $U_q(\h)$.
The algebras $U_q(\b_\pm)$ are equipped with $\Z$-grading assigning $\deg (e_{\pm\al})=1$ and $\deg(q^{\pm h_\al})=0$.
Then $U_q(\g_\pm)$ are graded subalgebras in $U_q(\b_\pm)$.

We fix the quasitriangular structure on $\U_q(\g)$ so that the R-matrix belongs
to an extended tensor product  $U_q(\b_+)\tp U_q(\b_-)$. It intertwines the coproduct
$\Delta(x)$ and its opposite $\Delta^{op}(x)$,
\be
\Ru \Delta(x)= \Delta^{op}(x)\Ru, \quad \forall x \in \U_q(\g).
\label{intertwiner}
\ee
In particular, $ \Ru (q^{h_\al}\tp q^{h_\al})= (q^{h_\al}\tp q^{h_\al})\Ru$ and
$
\Ru(e_{\al}\tp q^{h_{\al}} + 1\tp e_{\al})=(q^{h_{\al}}\tp e_{\al} + e_{\al}\tp 1)\Ru
$
for all $\al \in \Pi^+$.
The explicit expression for $\Ru$ can be found in \cite{CP}, Theorem 8.3.9.
 Let $h\dot{\tp} h=\sum_{s}h_s\tp h_s \in \h\tp\h$ denote the
inverse of canonical inner product, where the  summation is done over an orthogonal basis $\{h_s\}_{s=1}^{\rk\g} \subset \h$. Then
\be
\Ru= q^{ h\dot{\tp} h }\prod_{\bt\in \Rm^+} \exp_{q_\bt}\{(q-q^{-1})(e_\bt\tp e_{-\bt} )\},
\label{Rmat}
\ee
where  $\exp_{q}(x )=\sum_{k=0}^\infty q^{\frac{1}{2}k(k-1)}\frac{x^k}{[k]_q!}$. The product is ordered in accordance with
the normal ordering in $\Rm^+$ relative thePBW basis, see \cite{CP} for details.

It follows that
$q^{- h\dot{\tp} h }\Ru$ belongs to the tensor product $U_q(\g_+)\tp U_q(\g_-)$ completed
in the topology relative to the natural grading (inverse limit).
Introduce the tensor $\Fc=\frac{1}{q-q^{-1}}(q^{- h\dot{\tp} h }\Ru-1\tp 1)\in U_q(\g_+)\tp U_q(\g_-)$.
It plays the central role in our theory due to
\be
[1\tp e_{\al},\Fc] +(e_{\al}\tp q^{-h_{\al}})\Fc- \Fc (e_{\al}\tp q^{h_{\al}}) =
e_{\al}\tp [h_{\al}]_q, \quad \forall \al \in \Pi^+.
\label{key_id}
\ee
It is an immediate consequence of (\ref{intertwiner}).

Remark that $\Fc$ has the classical limit $\Fc_0=\sum_{\bt \in R^+}e_\bt\tp e_{-\bt}$ as $\hbar =\log q \to 0$. The intertwining relation (\ref{key_id})
has the obvious limit
\be
[1\tp e_{\al},\Fc_0]+[e_{\al}\tp 1,\Fc_0] = e_{\al}\tp h_{\al}, \quad \forall \al \in \Pi^+.
\nn
\ee
This makes our method directly applicable to the classical universal enveloping algebras with $\Fc_0$ in place of $\Fc$.

\section{Posets associated with representations}
In this section we introduce a partial ordering on a basis of a fixed $U_q(\g)$-module.
This data  will be used in  construction
of the inverse Shapovalov form in what follows.

Let $\Gm$ denote the root lattice, $\Gm=\Z\Pi^+$, and $\Gm^+$ the semigroup $\Z_+\Pi^+\subset \Gm$.
Consider a weight module $V$ over $U_q(\g)$ with the set of weights $\La(V)\subset \h^*$.
It particular, $V$ is $U_q(\h)$-diagonalizable and its weight spaces are finite dimensional.
Fix a weight basis $\{v_i\}_{i\in I}\subset V$ and let $\ve_i\in \La(V)$ be the weight of $v_i$, $i\in I$.

We introduce a partial ordering on $\La(V)$ by writing $\nu \geqslant  \mu$ if $\nu -\mu \in \Gm_+$.
 The ordering on $\La(V)$ induces an ordering on $I$, which we denote by the same symbol:
for $i,j\in I$ we write $i>j$ if and only if $\ve_i>\ve_j$.
It is convenient to use the language of Hasse diagram $\Hc$ associated with the poset $I$:
the arrows are marked with simple roots directed towards superior nodes.


The subset of all pairs $(i,j)$ such $\ve_i-\ve_j=\al \in \Gm^+$
will be denoted by $P(\al)$.
The pair $(i,j)$ is called simple if $\ve_i-\ve_j=\al \in \Pi^+$.
Let $\pi$ denote the representation homomorphism, $\pi\colon U_q(\g)\to \End(V)$. We can write
$
\pi(e_\al)=\sum_{(l,r)\in P(\al)}\pi^\al_{lr}e_{lr},
$
where $\pi^\al_{ij}$ are the entries of matrix $\pi(e_\al)$ in the basis $\{v_i\}_{i\in I}$ and $e_{ij}\in \End(V)$ are the standard matrix units, $e_{ij}v_k=\dt_{jk}v_i$.

For all $\al \in \h^*$ put
\be
\eta_{\al}=h_\al+(\al, \rho)-\frac{1}{2}||\al||^2 \in \h\op \C,
\ee
where  $\rho$ is the half-sum of all positive roots.
Put $\eta_{ij}=\eta_{\ve_i-\ve_j}$ for each pair $i,j\in I$ such that $i> j$.
\begin{lemma}
\label{thetas}
Suppose that $\al \in \Pi^+$ and $(l,r)\in P(\al)$. Then $\eta_{lr}=h_\al$. If $j\in I$ is such that $r> j$, then $\eta_{lj}-\eta_{rj}=h_\al+(\al,\ve_j-\ve_r)$.
\end{lemma}
\begin{proof}
It is known that $(\al, \rho)=\frac{1}{2}||\al||^2$ for all $\al \in \Pi^+$, hence the first statement. Further,
$$
\eta_{lj}-\eta_{rj}=h_\al+\frac{1}{2}||\al||^2+\frac{1}{2}||\ve_j-\ve_r||^2-\frac{1}{2}||\ve_j-\ve_r-\al||^2=h_\al+(\al,\ve_j-\ve_r),
$$
as required.
\end{proof}
In what follows, we have to extend the Cartan subalgebra to its localization, $\hat U_q(\h)$
over the multiplicative system generated by $[\sum_{\al \in \Pi^+} m_\al h_\al+k_\al \frac{(\al,\al)}{2}]_q$ with
 integer $m_\al,k_\al$. We extend all quantum groups containing $U_q(\h)$ accordingly, in the straightforward way,
and mark them with hat, e.g. $\hat U_q(\g)$ {\em etc}.
The apperance of $h_{\ve_i}\in \h$, $\ve_i\in \La(V)$, formally requires an appropriate extension of $\hat U_q(\h)$.
However, the resulting formulas depend only on the differences $\ve_i-\ve_j\in \Gm$, and such an extension will be suppressed.

\section{Construction of intertwiner}
We call a sequence  $\vec m=(m_1,\ldots, m_k)$  a route from $\min \vec m=m_k$  to $\max \vec m=m_1$ if
$m_1> \ldots > m_k$. Let $|\vec m|=k$ designate the number of nodes in $\vec m$.
The set of all routes from $j$ to $i$ will be denoted
by $(i \leftarrow j)$.
A route is called path if it is maximal, i.e. all pairs $(m_s,m_{s+1})$ are simple.
If $\vec m$ is a path, then $|\vec m|-1$ is its length, i.e. the number of arrows. It is equal to
the number of simple positive roots constituting the weight $\ve_{\max \vec m}-\ve_{\min \vec m}$.
We use the notation
$\vec m > \vec n$ if $\min \vec m> \max \vec n$,
and $\vec m \geqslant  \vec n$ if $\max \vec m= \min \vec n$ (this relation is transitive on routes).
In this case we can construct a route $(\vec m, \vec n) \in (\max \vec m \leftarrow \min \vec n )$ by taking the union
of nodes in the prescribed order.

Given a route $\vec m=(m_1,\ldots, m_k)$  we define
$
f_{\vec m}=f_{m_1m_2}f_{m_2m_3}\ldots f_{m_{k-1}m_k},
$
where $f_{ij}\in U_q(\g_-)$ are the entries of the matrix $F=(\pi\tp \id)(\Fc)$.
Fix an ordered pair $(i,j)$, $i> j$, and consider a right $\hat U_q(\h)$-module $\Phi_{ij}$ freely generated by
all $f_{\vec m}$ such that  $ \vec m \in (i\leftarrow j)$.

For each simple pair $(l,r)\in P(\al)$ we introduce an operator $\prt_{lr}\colon \Phi_{ij}\to \hat U_q(\g)$
as follows.
Suppose that $(\vec \ell,l,r,\vec \rho)\in (i\leftarrow j) $ is a route.
Then put
$$
\begin{array}{rrccc}
\partial_{lr}f_{(\vec\ell,l)}f_{lr}f_{(r,\vec \rho)}&=&f_{(\vec\ell,l)}f_{(r,\vec \rho)}[\eta_{lj}-\eta_{rj}]_q,\\
\partial_{lr}f_{(\vec\ell,l)}f_{(l,\vec \rho)}&=&-f_{(\vec\ell,l)}f_{(r,\vec \rho)}q^{-\eta_{lj}+\eta_{rj}},
\\
\partial_{lr}f_{(\vec\ell,r)}f_{(r,\vec \rho)}&=&f_{(\vec\ell,l)}f_{(r,\vec \rho)}q^{\eta_{lj}-\eta_{rj}}.
\end{array}
$$
If  $\{l,r\}\cap \vec m=\emptyset$ or $\vec m\cup \{l,r\}$ is not a route from $(i\leftarrow j)$,  then we set $\prt_{lr}f_{\vec m}=0$.
Having defined $\prt_{lr}$ on the generators, we extend it to the entire $\Phi_{ij}$ by $\hat U_q(\h)$-linearity.

\begin{propn}
Let $p\colon \Phi_{ij}\to \hat U_q(\g)$ be the natural homomorphism of right $\hat U_q(\h)$-modules.
Then $e_\al \circ p= \sum_{(l,r)\in P(\al)}\pi^\al_{lr}\prt_{lr} \mod \hat U_q(\g)\g_+$ for each $\al\in \Pi^+$.
\label{localization}
\end{propn}
\begin{proof}
First of all, observe that both operators commute with the right $U_q(\h)$-action modulo $U_q(\g)\g_+$, therefore
we can prove the equality on the basis $\{f_{\vec m}\}\subset  \Phi_{ij}$ replacing the left-hand side with the commutator $f_{\vec m}\mapsto [e_\al, f_{\vec m}]$, on identification of $f_{\vec m}$ with its image $p(f_{\vec m})$.
In the fixed basis of $V$, the intertwining relation (\ref{key_id}) translates to
\be
[e_\al, f_{ij}]&=&
\left\{\begin{array}{lrr}
 \pi^\al_{ij}[h_\al]_q,
 &
\quad (i,j)\in P(\al),
\\[10pt]
 \sum_{(l,r)\in P(\al)}\bigl( \dt_{jr}f_{il}\pi^\al_{lr}q^{h_\al}-q^{-h_\al}\dt_{li}\pi^\al_{lr}f_{rj}\bigr),
& (i,j)\not \in P(\al).
\end{array}
\right.
\label{i-rel-bas}
\ee
Denote by $B(\al)$ the union of all  $\{l,r\}\subset I$ such that $(l,r)\in P(\al)$.
Fix a route $\vec m \in (i\leftarrow j)$.
If  $\vec m \cap B(\al)=\emptyset$ then $\sum_{(l,r)\in P(\al)}\pi^\al_{lr}\prt_{lr}f_{\vec m}=0$ by construction.
At the same time, $[e_\al,f_{km}]=0$ for all  $f_{km}$ entering $f_{\vec m}$, hence $e_\al f_{\vec m}=0 \mod \hat U_q(\g)\g_+$.

Suppose that  $\vec m \cap B(\al)\not=\emptyset$.
The commutator with $e_\al$ differentiates the product $f_{\vec m}$ making it into a sum of terms
where $e_\al$ acts on just one factor:
$[e_\al, f_{\vec m}]=\sum f_{(\vec \ell, a)}[e_\al,f_{ab}]f_{(b,\vec \rho)}$.
If $[e_\al,f_{ab}]\not=0$, then either $a$ or $b$ belongs to $B(\al)$. There are three possibilities:
$(a,b) =(l,r)$ or $(a,b) =(l,b), (a,r)$ , where $(l,r)\in P(\al)$ and $r>b$ and $a> l$.
Then
$$
[e_{\al},f_{lr}]=\pi^\al_{lr}[h_\al]_q, \quad [e_{\al},f_{lb}]=-\pi^\al_{lr}q^{-h_\al}f_{rb},\quad [e_{\al},f_{ar}]=\pi^\al_{lr}f_{al}q^{h_\al},
$$
in these three cases.
We can assume that $(l,r)\cup \vec m$ is a route, since otherwise
 $r\not >b \Rightarrow f_{rb}=0$,  $a\not >l\Rightarrow f_{al}=0$, and the last two commutators
turn zero along with  $\prt_{lr}f_{\vec m}$.
Then the Cartan factors are pushed to the rightmost position and acquire the shift
 $h_\al\mapsto h_\al + (\al,\ve_j-\ve_r)=\eta_{lj}-\eta_{rj}$, by Lemma \ref{thetas}.
 Therefore,
the corresponding summand $f_{(\vec \ell, a)}[e_\al,f_{ab}]f_{(b,\vec \rho)}$ is equal to $\pi^\al_{lr}\prt_{lr}f_{\vec m}$.
This proves the equality on the $U_q(\h)$-basis of $\Phi_{ij}$.
\end{proof}
Fix $j$ and introduce $A_i^j\in \hat U_q(\h)$ for all $i>j$ by
$$
A_{i}^{j}=\varphi(-\eta_{ij}),\quad \varphi(x)=\frac{q^{-x}}{[x]_q}.
$$
Note that $A_{i}^{j}$ depends only on the weight difference $\ve_i- \ve_j$.
For  $\vec m=(m_1,\ldots, m_k)$, $\vec m > j$, put $A_{\vec m}^{j}=A_{m_1}^{j} \ldots A_{m_k}^{j}.$
Fix a simple pair $(l,r)$ and suppose that $i\geqslant \vec\ell >l> r> \vec\rho > j$
for some $\vec \ell$ and $\vec \rho$. Define an open $(l,r)$-chain to be the sum
\be
f_{(\vec \ell,r)}f_{(r,\vec \rho,j)}A_{(\vec \ell,r,\vec \rho)}^{j} +f_{(\vec \ell,l)}f_{lr}f_{(r,\vec \rho,j)}A_{(\vec \ell,l,r,\vec \rho)}^{j} +f_{(\vec \ell,l)}f_{(l,\vec \rho,j)}A_{(\vec \ell,l,\vec \rho)}^{j}\quad \in \Phi_{ij},
\label{3-chain}
\ee
assuming  $\vec \ell\not =\emptyset \not =\vec \rho$.
Also, we define boundary $(l,r)$-chains as
$$
f_{(\vec \ell,j)}A_{\vec \ell}^{j} +f_{(\vec \ell,l)}f_{lj}A_{(\vec \ell,l)}^{j}, \quad f_{ir}f_{(r,\vec \rho,j)}A_{(r,\vec \rho)}^{j} +f_{(i,\vec \rho,j)}A_{\vec \rho}^{j}\quad \in \Phi_{ij},
$$
where  $\vec \ell\not =\emptyset $ and $\vec \rho\not =\emptyset$. In particular, the pair
$(i,j)$ is assumed non-simple.
\begin{lemma}
The operator $\prt_{lr}$ annihilates open $(l,r)$-chains.
\label{open_chains}
\end{lemma}

\begin{proof}
Applying the local operator $\prt_{lr}$ to (\ref{3-chain}) we get
$$
f_{(\vec \ell,l)}f_{(r,\vec \rho,j)}\Bigl(q^{\eta_{lj}-\eta_{rj}}A_{(\vec \ell,r,\vec \rho)}^{j} +[\eta_{lj}-\eta_{rj}]_qA_{(\vec \ell,l,r,\vec \rho)}^{j} -q^{-\eta_{lj}+\eta_{rj}}A_{(\vec \ell,l,\vec \rho)}^{j}\Bigr).
$$
Division by $A_{(\vec \ell,l,r,\vec \rho')}^{j}$
of the factor in the brackets gives
$
\frac{q^{\eta_{lj}-\eta_{rj}}}{A_{l}^{j}} +[\eta_{lj}-\eta_{rj}]_q -\frac{q^{-\eta_{lj}+\eta_{rj}}}{A_{r}^{j}}=0.
$
Therefore this factor is zero, and the proof is complete.
\end{proof}

\begin{lemma}
The local operators $\prt$ act on boundary chains by
$$
 f_{ii'}f_{(i',\vec \rho,j)}A_{(i,i',\vec \rho)}^j +f_{(i,\vec \rho,j)}A_{(i,\vec \rho)}^j
\stackrel{\prt_{ii'}}{\longmapsto}
 -q^{h_{\ve_i}-h_{\ve_{i'}}}f_{(i',\vec \rho,j)}A_{(i',\vec \rho)}^{j},
\quad
f_{(i,\vec \ell,j)}A_{\vec \ell}^{j} +f_{(i,\vec \ell,j')}f_{j'j}A_{(i,\vec \ell,j')}^{j}
\stackrel{\prt_{j'j}}{\longmapsto} 0,
$$
where $i> \vec \ell> j'$ and $i'> \vec \rho > j$.
\label{boundary_chains}
\end{lemma}

\begin{proof}
Put $\al=\ve_{j'}-\ve_j$. Application of $\prt_{j'j}$ to the right chain gives
$$
f_{(i,\vec \ell,j')}A_{(i,\vec \ell)}^{j}(q^{h_\al}+[h_\al]_qA_{j'}^{j})
=f_{(i,\vec \ell,j')}A_{(i,\vec \ell)}^{j}\frac{[h_\al-\eta_{j'j}]_q}{[\eta_{j'j}]_q}=0,
$$
since $\eta_{j'j}=h_\al$, by Lemma \ref{thetas}. To prove the left formula, we put $\al=\ve_i-\ve_{i'}$. Application of $\prt_{ii'}$ gives
$$
f_{(i',\vec \rho,j)}([\eta_{ij}-\eta_{i'j}]_qA_{i'}^{j} -q^{-\eta_{ij}+\eta_{i'j}})A_{(i,\vec \rho)}^{j}.
$$
The expression in the brackets is equal to
$-\frac{[\eta_{ij}]_q}{[\eta_{i'j}]_q}=-\frac{A^j_{i'}}{A^j_{i}}q^{\eta_{ij}-\eta_{i'j}}
=-\frac{A^j_{i'}}{A^j_{i}}q^{h_{\al}+(\al,\ve_j-\ve_{i'})}.
$
Observe that $\frac{A^j_{i'}A_{(i,\vec \rho)}^{j}}{A^j_{i}}=A_{(i',\vec \rho)}^{j}$.
Now recall that the weight of $f_{(i',\vec\rho,j)}$ is $\ve_j-\ve_{i'}$ and push the factor $q^{h_{\al}}=q^{h_{\ve_{i}}-h_{\ve_{i'}}}$ to the leftmost position.
This completes the proof.
\end{proof}

Define a matrix $\hat F\in \End(V)\tp \hat U_q(\b_-)$ by
\be
\hat F=1\tp 1+\sum_{i>j}e_{ij}\tp \hat f_{ij}, \quad \hat f_{ij}=\sum_{i\geqslant \vec m> j}f_{(\vec m, j)}A_{\vec m}^{j}.
\label{dynamization}
\ee
We can also define $\hat f_{ii}=1$ and $\hat f_{ij}=0$ if $i\not\geqslant j$ making
$\hat F=\sum_{i,j\in I}e_{ij}\tp \hat f_{ij}$.
Since all weight subspaces are finite dimensional, the number of routes in $(i\leftarrow j)$ is finite and
summation is well defined.
\begin{propn}
For all $x\in U_q(\g_+)$,
$
\Delta(x)\hat F= \eps(x)\hat F \mod
\End(V)\tp \hat \U_q(\g)\g_+.
$
\label{hat F}
\end{propn}
\begin{proof}
We need to prove
$
e_\al \hat f_{ij}= -\sum_{r\in I}\pi(e_\al)_{ir}q^{h_{\al}} \hat f_{rj} \mod
\hat U_q(\g)\g_+
$
for all $\al\in\Pi^+$ and  $i,j\in I$, $i\geqslant j$.
First suppose that the pair $(i,j)$ is not simple.
Following Proposition \ref{localization} we reduce the left-hand side to the localization
$\sum_{(l,r)\in P(\al)}\pi^\al_{lr}\prt_{lr} \hat f_{ij} = e_\al \hat f_{ij}\mod \hat U_q(\g)\g_+$.
Evaluating $\prt_{lr} \hat f_{ij}$ we  retain only those routes $(\vec m,j)$  in summation (\ref{dynamization}) which
intersect with  $(l,r)$ and extend to a path from $j$ to $i$ passing through $(l,r)$. Summation over such routes can be arranged in
either a) open three-chains if $i>(l,r)>j$  or
b) right/left boundary chains if, respectively, $j=r$ or $i=l$. By Lemmas \ref{open_chains} and \ref{boundary_chains}, only the last case contributes to $\prt_{lr} \hat f_{ij}$,
and the contribution is equal to $-q^{h_{\al}}\hat f_{rj}$.

If the pair $(i,j)$ is simple and $\al=\ve_i-\ve_j$, then $\hat f_{ij}=-f_{ij}\frac{q^{\eta_{ij}}}{[\eta_{ij}]_q}
=-f_{ij}\frac{q^{h_\al}}{[h_\al]_q}$. So, by (\ref{i-rel-bas}),  $e_\al\hat f_{ij}=-\pi(e_\al)_{ij}q^{h_\al}$ modulo $\hat U_q(\g)\g_+$. Finally, the obvious equality $
e_\al \hat f_{ij}= 0\mod \hat U_q(\g)\g_+$ for $i=j$ completes the proof.
\end{proof}
\begin{remark}
The matrix $\hat F$ can be obviously defined for every weight module $V$ over $U_q(\b_+)$. Moreover, it
depends only on the structure of $U_q(\g_+)$-module on $V$. For example, it coincides for all Verma modules of lowest weight, see Section \ref{Sec_Shap_form}.
\end{remark}
\section{Universal matrix $\hat \Fc$}

In this section, we introduce a "universal matrix" $\hat \Fc$ producing $\hat F$ in representations.
Upon passing to the $\C[\![\hbar ]\!]$-algebra $\hat U_\hbar(\g)$, set $h\cdot h=\sum_{s=1}^{\rk \g}h_s^2\in U_\hbar(\h) $, where $\{h_s\}_{s=1}^{\rk \g}$ is an orthonormal basis in $\h$.
Introduce $d\in U_\hbar(\h)$ by
\be
d= \frac{1}{2}h\cdot h + h_\rho.
\label{d}
\ee
The commutator $[d, \cdot\> ]$ acts on the subspace of weight $-\mu$ in $\hat U_\hbar(\b_-) $
as multiplication by $-\eta_\mu$ on the right. Therefore,
$\frac{q^{2[d, \cdot\> ]}-\id}{q-q^{-1}}$ is invertible on $\hat U_q(\b_{-})\g_-$, and its  inverse $\varphi([d, \cdot\> ])$
acts as right multiplication by $\varphi(-\eta_\mu )$ on each subspace of weight $-\mu<0$.
Then $\varphi(D)$, with $D=\id \tp [d, \cdot\> ]$, is well defined
on the subspace of zero weight in $U_q(\g_{+})\g_+\tp \hat U_q(\b_-)\g_-$.
We do not need $\C[\![\hbar]\!]$-extension any longer.

Define tensors $\Fc^{(k)}\in  U_q(\g_+)\tp \hat U_q(\b_-) $, $k\in\Z_+$, as follows. Put $\Fc^{(0)}= 1\tp 1$, and,  by induction,
$\Fc^{(k+1)}=\varphi(D)\bigl( \Fc\Fc^{(k)}\bigr)$ for $k\geqslant 0$. Define
\be
\hat\Fc=\sum_{k=0}^\infty\Fc^{(k)}\in U_q(\g_+)\tp \hat U_q(\b_-),
\label{geom F}
\ee
where the tensor product is completed in the topology relative to the natural grading in $U_q(\g_+)\tp \hat U_q(\b_-)$
(projective limit).
 \begin{propn}
For every representation $\pi\colon U_q(\b_+)\to \End(V)$ on a weight module $V$, the matrix $(\pi\tp \id)(\hat \Fc)\in  \End(V)\tp \hat  U_q(\b_-)$ coincides with
$\hat F$.
\label{univ_project}
\end{propn}
\begin{proof}
Both matrices have their $(i,j)$-entries
equal to $1$ if $i=j$ and to $0$ if $i<j$ or incomparable. Now suppose that $i>j$ and write
$$
e_{ij}\tp f^{(k+1)}_{ij}=
\sum_{m_1\in I}\ldots \sum_{m_k\in I}
 \varphi(D)\bigl((e_{i,m_1}\tp f_{i,m_1})\varphi(D)\bigl(\ldots \varphi(D)(e_{m_k,j}\tp f_{m_k,j})\bigr),
$$
where $f^{(k+1)}_{ij}$ are the entries of the matrix $(\pi\tp \id)(\Fc^{(k+1)})=\sum_{i,j\in I}e_{ij}\tp f^{(k+1)}_{ij}$.
In the above summation, we retain only the terms with $i>m_1>\ldots >m_k>j$, since
one of the $f$-factors turns zero otherwise.
Observe that $\varphi(D)(e_{ij}\tp X_{ij}) =e_{ij}\tp X_{ij}A^j_i $ for every element $X_{ij}\in \hat U_q(\b_-)$
of weight $\ve_j- \ve_i$. Then
one can easily check that each summand is equal to
$$
(e_{i,m_1}\tp f_{i,m_1})\ldots (e_{m_k,j}\tp f_{m_k,j})A^j_{(i,\vec m)}
 =
e_{ij}\tp f_{(i,\vec m,j)}A^j_{(i,\vec m)}.
$$
Varying $\vec m =(m_1,\ldots, m_k)$
within  $i>\vec m > j$ including  $\vec m =\emptyset$ we cover all routes in $(i\leftarrow j)$,
$$
 \sum_{k=0}^\infty e_{ij} \tp f_{ij}^{(k)}= e_{ij} \tp \hat f_{ij},
$$
as required.
\end{proof}
\noindent
\begin{corollary}
\label{natural}
The tensor $\hat F$ is independent of  basis in $V$.
Let $V$ and $V'$ be weight $U_q(\b_+)$-modules
and $\hat F\in V\tp  \hat U_q(\b_-)$, $\hat F'\in V'\tp \hat U_q(\b_-)$.
Then $(\phi\tp 1)\hat F=\hat F '(\phi\tp 1)$ for every  $U_q(\g_+)$-homomorphism $\phi\colon V\to V'$.
\end{corollary}
\begin{proof}
Readily follows from the existence of the universal matrix $\hat \Fc$.
\end{proof}
It is also convenient to refine the ordering on $I$.
We write $i\succ j$ if and only if $i>j$ and there is a sequence $i=m_1>\ldots> m_k>j$ such that
$f_{m_l,m_{l+1}}\not =0$ for all $k=1,\ldots, k-1$ (it is transitive since $U_\g(\g_-)$ has no zero divisors).
In particular, there is a sequence of simple roots $\bt_1,\ldots, \bt_{k-1}\in \Pi^+$
and a sequence of indices $i=m_1,\ldots,m_{k}=j\in I$ such that
$\pi^{\bt_1}_{m_1,m_2}\ldots \pi^{\bt_{k-1}}_{m_{k-1},i_{k}}\not =0$.
 The Hasse diagram $(\Hc,\succ)$ is a subdiagram in $(\Hc,>)$. It is generally not full, i. e. not all arrows
 connecting nodes in $(\Hc,>)$ are in $(\Hc,\succ)$.
\begin{lemma}
The matrices $(\hat F,{>})$ and $(\hat F,{\succ })$ are equal.
\end{lemma}
\begin{proof}
Observe that  $f_{ij}=\Tr(e_{ji}\Fc_1)\Fc_2=0$ for $i>j$ but $i\not \succ j$, where $\Fc_1\tp \Fc_2=\Fc$ is the Sweedler notation.
Therefore, the product $f_{\vec m}$ turns zero if $\vec m$ is a route with respect to $>$ but not $\succ$.
Finally, the coefficients $A^j_i$ depend only on weights, hence the proof.
\end{proof}
Next we show that the matrix associated with a factor-module can be obtained from the the covering module.
\begin{propn}
\label{Nat}
Suppose that $V$ and $V'$ are two weight $U_q(\b_+)$-modules and $\phi\colon V\to V'$
is a surjective $U_q(\g_+)$-homomorphism. Choose a basis $\{v_i\}_{i\in I} \in V$ such that $\phi(v_i)=0$ for $i\in \bar I'=I\backslash I'$ and
$\phi(v_i)=w'_i$ for $i\in I'$ forms a basis in $V'$. Then $\hat f_{ij}=\hat f_{ij}'$ for $i,j\in I'$.
\end{propn}
\begin{proof}
Take $\succ$ for the partial ordering in $I$ and $I'$.
We have an embedding $\Hc'\subset \Hc$ of the Hasse diagrams. Moreover, $\Hc'$ is a full subdiagram, since
$f_{ij}\not =0$ for all $i,j\in I'$ simultaneously in $\Hc$ and $\Hc'$. Furthermore, all $\Hc$-routes between $i,j\in \Hc'$ are routes in $\Hc'$.
Indeed, let $\bar V'\subset V$ be the kernel of $\phi$ spanned by $w_i$, $i\in \bar I'$.
Let $\vec m$ be a route from $j$ to $i$ in $\Hc$. If  a simple pair $(l,r)\in \vec m$ is such that $l\in \bar I'$ and $r\in I'$, then all nodes $m \succcurlyeq l$ including $i$ belong to $\bar I'$, which is  a contradiction. Therefore $\vec m$ is in $\Hc'$.
\end{proof}

\section{Singular vectors and inverse Shapovalov form}
\label{Sec_Shap_form}
In this section we show that the tensor $\hat \Fc\in U_q(\g_+)\tp \hat U_q(\b_-)$ is a lift of the inverse Shapovalov form.
Such an indication follows from Proposition \ref{hat F}. To complete the task, we need to
prove a similar identity for the negative Chevalley generators. Instead, we
employ  the fact that the inverse Shapovalov form is producing singular vectors in
the tensor product of $U_q(\g)$-modules serving as an intertwiner.

For all weights $\la\in \h^*$ consider one-dimensional representations $\C_{\la}$ of $U_q(\h)$ defined by $q^{h_\al}\mapsto q^{(\la,\al)}$.
Extend them  to representations of $U_q(\b_\pm)$ by nil on $U_q(\g_\pm)$ and
define $U_q(\g)$-modules
$$
M_\la=\hat U_q(\g)\tp_{\hat U_q(\b_+)}\C_{\la}, \quad M_\la^*=\hat U_q(\g)\tp_{\hat U_q(\b_-)}\C_{-\la}.
$$
Let $1_{\la}\in M_\la$ and $1^*_\la\in M_\la^*$ denote their canonical generators.
The BPW property of $U_q(\g)$ implies that $M_\la^*$ is free as a $U_q(\g_+)$-module.
A weight basis $\{e_i\}_{i\in I}\subset U_q(\g_+)$
is parameterized by $I\simeq \Z_+^{\dim \g_+}$, with
 $e_0=1$.  Then $\{e_i1^*_{\la}\}_{i\in I}$ is a basis in $M_\la^*$.

Recall that a vector $u$ in a $U_q(\g)$-module is called singular if $e_\al u=\eps(e_\al)u= 0 $ for all $\al\in \Pi^+$.
It is known that, in the case of finite dimensional $V$ and generic $\la$ there is a singular vector $u_j\in V\tp M_\la$ of weight $\la+ \ve_j$  for all $j\in I$. For the reader's convenience we present the construction  via the inverse of a special
invariant pairing $M_\la \tp M^*_\la\to \C$. We call it Shapovalov form because it is equivalent to the usual contravariant form on
$M_\la$, \cite{Jan}.

The form in question is determined by the requirements
$\langle 1_\la, 1^*_\la\rangle =1$  and $\langle x 1_\la, 1^*_\la\rangle = \langle 1_\la,\gm(x)1_\la^*\rangle$ for all $x\in U_q(\g)$.
It is non-degenerate if and only if $M_\la$ and $M^*_\la$ are irreducible. The form is known to be non-degenerate for generic $\la$. Let $\S^\la\in M_\la^*\tp M_\la$ be its inverse and write it as
$\S^\la=\sum_{i\in I} e_i 1^*_\la\tp f_i1_\la$,
where $f_i=f_i(\la)\in  U_q(\g_-)$. Then $\{f_i1_\la\}_{i\in I}$ forms the dual basis to $\{e_i1_\la^*\}_{i\in I}$ for
those $\la$ where the inverse is defined, and $\S^\la$ is independent of the choice of $\{e_i\}_{i\in I}$.
It readily follows from the invariance of $\S^\la$ that $\sum_{i\in I}e_i v_j\tp f_i1_\la$ is a singular vector in $V\tp M_\la$.
\begin{propn}
Suppose that the modules $M_\la$, $M_{\la+\ve}$, $\forall \ve \in \La(V)$, are irreducible.
Then $V\tp M_\la \simeq \oplus_{j\in I} M_{\la+\ve_j}$.
\end{propn}
\begin{proof}
Let $V^*$ be the right dual to $V$ and let $\{v^j\}_{j\in I}$ be the dual basis to
$\{v_j\}_{j\in I}$, $\langle v_j,v^i\rangle=\dt_{ij}$. Observe that $v^{j}$ has weight $-\ve_j$.
Since the Verma modules are irreducible, their Shapovalov forms are non-degenerate.
We have two module homomorphisms
$$
M_{\la+\ve_j}\to V\tp M_\la , \quad V\tp M_\la \to M_{\la+\ve_j},
$$
$$
1_{\la+\ve_j}\mapsto
\sum_{i\in I} e_i v_j\tp f_i 1_\la,
\quad v\tp 1_\la \to \sum_{i\in I}\langle v,  e_iv^j\rangle f_i1_{\la +\ve_j}.
$$
They amount to the homomorphisms
$$
\oplus_{j\in I} M_{\la+\ve_j} \to V\tp M_\la \to \oplus_{k\in I} M_{\la+\ve_k}.
$$
By irreducibility, the through map $ M_{\la+\ve_j}  \to M_{\la+\ve_k}$ is zero if $\ve_j\not =\ve_k$.
Otherwise it goes as
$$
1_{\la +\ve_j}\mapsto \sum_{i\in I} e_iv_j\tp f_i1_\la \to
\sum_{i\in I} \sum_{l\in I}\langle e_iv_j, f_i^{(1)} e_lv^k\rangle f_i^{(2)} f_l1_{\la +\ve_k}=\langle v_j, v^k \rangle 1_{\la +\ve_k}= \dt_{jk}1_{\la +\ve_k},
$$
where $f_i^{(1)}\tp f_i^{(2)}$ is the Sweedler notation for $\Delta(f_i)$.
Indeed, since the weights of $e_if^{(1)}_i$ and $e_l$ are non-negative,
only the terms with $i=0$, $l=0$ contribute to the sum.
Therefore the composition map is identical, and  the left map is injective.
Comparing the dimensions of the weight subspaces one can prove that it is an surjective as well.
\end{proof}

It follows, for generic $\la$,  that the map $v\mapsto \sum_{i\in I} e_i v\tp f_i1_\la$ is a $U_q(\h)$-linear bijection of $V$ onto the set of singular vectors in $V\tp M_\la$.
At the same time, the
tensors $\hat F(v_j\tp 1_\la)=\sum_{j\in I} v_i\tp \hat f_{ij}1_\la$
are singular vectors of weight $\la+\ve_j$, by  Proposition \ref{hat F}.
They are well defined provided
the weight $\la$ is not a solution of the equation
$
q^{2(\la+\rho,\al)-||\al||^2}=1
$
for all $\al =\ve_i-\ve_j$ such that $i\succ j$. Let $\hat\h^*_{reg}(V)$ denote the set of such weights and $\h^*_{reg}(V)\supset \hat\h^*_{reg}(V)$
the set of $\la$ admitting analytical continuation of $\hat F$.
It readily follows from Proposition \ref{univ_project} that, under the isomorphism $U_q(\g_+)\simeq M^*_\mu$, the matrix
$\hat F\in \End(M_\mu^*)\tp \hat  U_q(\g_-)$ is independent of $\mu$.
Moreover, $\hat F (v\tp 1_\la) = \sum_{i\in I}e_i v \tp \hat f_{i0}1_\la=\hat \Fc(v \tp 1_\la) \in M_\mu^*\tp M_\la$ for all $v\in M^*_\mu$.
\begin{thm}
Suppose that $q^{2(\la+\rho,\al)-m||\al||^2}\not =1$ for all $\al \in \Rm^+$ and $m\in\N$. Then the tensor $ \hat \Fc (1^*_\la\tp 1_\la)$ is
the inverse of the Shapovalov form $\S^\la$.
\end{thm}
\begin{proof}
For fixed $q$, the vectors $f_i$ are rational in $q^{(\la,\al)}$, $\al \in \Pi^+$.
Since $\S^\la$ is independent of basis $\{e_i\}_{i\in I}$, it is sufficient to prove the statement for $\la \in \hat\h^*_{reg}(M^*_\la)$ and do analytical continuation to  $ \h^*_{reg}(M^*_\la)$, where the module $M_\la$ is still irreducible.
According to \cite{DCK}, it is the set of $\la$ such that $q^{2(\la+\rho,\al)-m||\al||^2}\not =1$ for all $m\in \N$ and all $\al \in \Rm^+$.

Given a finite dimensional irreducible $U_q(\g)$-module $V$ with the lowest vector $v_0$ of weight $-\mu$,
there is a unique singular vector $\sum_{i\in I}e_i v_0\tp f_i1_\la$ of weight $-\mu+\la$ in $V\tp M_\la$, up to a scalar multiplier.
 On the other hand, it is equal to $\hat \Fc (v_0\tp 1_\la)$,
therefore they coincide up to a scalar. Since this holds for all finite dimensional modules $V$,
we have $\sum_{i\in I}e_i \tp f_i1_\la \sim \sum_{i\in I}  e_{i}\tp \hat f_{i0}1_\la$, which is obviously an equality,
by comparison at $i=0$.
\end{proof}

The Shapovalov form  vanishes on proper submodules in $M^*_\la$ and $M_\la$ and can be projected to their quotients.
In general, a module of highest (lowest)  weight is irreducible if and only if the Shapovalov form has zero kernel on it.
Let $-\la$ be the lowest weight of $V$. Then $V$  is a quotient of a negative Verma module, $M_{\la}^*$, by a
submodule $\bar V$.
 The annihilator $V^\perp$ of $V$ in $M_\la$ is a submodule.
 Let $W$ denote the quotient $M_\la$ by $V^\perp$ and $w_0 \in W$, $v_0\in V$ the images of $1_\la$, $1_\la^*$, respectively.
The Shapovalov form factors through invariant pairing $W\tp V\to \C$.
\begin{propn}
Suppose that $\la\in \h_{reg}^*(V)$. Then $\hat \Fc(v_0\tp w_0)\in V\tp W$ is the Shapovalov inverse,
and the modules $V$, $W$ are irreducible.
\label{suffic_irred}
\end{propn}
\begin{proof}
Irreducibility follows from the first statement, so let us prove it.
Apply the homomorphism $M_\la^*\to V$ to the left tensor factor of $\hat \Fc$
and get $\hat F\in \End(V)\tp \hat U_q(\b_-)$. By the hypothesis, $\hat F(v_0\tp 1)$ is regular when its right factor specialized
at $\la$,
so  $\hat F(v_0\tp 1_\la)\in V\tp M_\la$ and $\hat F(v_0\tp w_0)\in V\tp W$ are well defined.
Identify $V$ with  a subspace in $M_\la^*$ complementary to $\bar V$.
By Proposition \ref{Nat}, $\hat F$ yields a non-degenerate pairing of $V$ with a subspace in $M_\la$. Therefore
$V$ is irreducible along with $W$. Projection to $W$ sends
$\hat F(v_0\tp 1_\la)$ to  $\hat F(v_0\tp w_0)$, the Shapovalov inverse of
the  non-degenerate form $W\tp V\to \C$.
\end{proof}
\noindent
Proposition \ref{suffic_irred} gives a sufficient condition on the highest (lowest) weight for a module to be irreducible.
\section{ABRR equation revisited}
Formula  (\ref{geom F}) readily implies that $\hat \Fc$ satisfies the linear identity
\be
\hat \Ru\hat \Fc = q^{2\tp d}(\hat \Fc)q^{-2\tp d},\quad \hat \Ru = q^{-h\dot \tp h}\Ru=1\tp 1+(q-q^{-1})\Fc.
\label{ABRR}
\ee
The tensor $\hat \Fc$ lies in the kernel of the operator $X\mapsto q^{-2D}(\hat \Ru X)-X$. The latter is continuous
in the topology on $U_q(\g_+)\tp \hat U_q(\b_-)$ relative to the decreasing  filtration induced by the grading.
Its diagonal part $q^{-2D}-\id$ is invertible when restricted to the zero weight subalgebra of positive degree. Therefore, $\hat \Fc $ is a unique solution of (\ref{ABRR}) with the zero degree component $1\tp 1$.

There is a unique tensor $\hat \S\in U_q(\g_+)\tp U_q(\g_-)\tp \hat U_q(\h)$
such that $(\hat \S_11_{\la}^*)\tp (\hat \S_2 1_\la) \>\la(\hat \S_3)=\S^\la$. The linear isomorphism $\iota'\colon U_q(\g_+)\tp \hat U_q(\b_-)\mapsto U_q(\g_+)\tp U_q(\g_-)\tp \hat U_q(\h)$ (an extension of tensor product required)
takes $\hat \Fc$ to $\hat \S$. Denote by $\Jc$ the image of $\hat \S$ under the linear isomorphism
$\iota''\colon a\tp b \tp h\mapsto a\tp bh^{(1)} \tp h^{(2)}\in U_q(\g_+)\tp \hat U_q(\b_-)\tp \hat U_q(\h)$. The composition $\iota=\iota''\circ \iota'
\colon U_q(\g_+)\tp \hat U_q(\b_-)\mapsto U_q(\g_+)\tp \hat U_q(\b_-)\tp \hat U_q(\h)$
is an algebra homomorphism.
Therefore $\Jc=\iota (\hat \Fc)$ satisfies the equation
\be
\hat \Ru \Jc = q^{2(1\tp d\tp 1 +1\tp h\dot \tp h)}  \Jc q^{-2(1\tp d\tp 1 +1\tp h\dot \tp h)},
\label{ABRR1}
\ee
since $\iota (1\tp d)= 1\tp d\tp 1 +1\tp h\dot \tp h+1\tp 1\tp d$ (the last summand cancels from (\ref{ABRR1})).
The tensors $\Jc$ and $\hat \S$ are expressed through each other by
$$
\Jc=\hat \S_1\tp \hat \S_2\hat \S_3^{(1)}\tp \hat \S_3^{(2)}, \quad
\hat \S=\Jc_1\tp \Jc_2\gm(\Jc_3^{(1)})\tp \Jc_3^{(2)},
$$
where the Sweedler notation $\Delta(x)=x^{(1)}\tp x^{(2)}$ is used for the coproduct.
We also have $\hat \Fc=\Jc_1\tp \Jc_2 \eps(\Jc_3)$, and the homomorphism $\id\tp \id\tp \eps$ is  inverse to $\iota$.
The tensor $\Jc$ is known to be a dynamical twist, \cite{ES}.
The algebra automorphism of $U_q(\h)$ determined by $q^{h_\al}\mapsto q^{h_\al-(\rho,\al)}$, $\al\in \Pi^+$,
 preserves the dynamical twist identity and takes (\ref{ABRR1}) to
the ABRR equation of \cite{ABRR}. The image of $\Jc$ under this transformation coincides with the twist of \cite{ABRR} due
to its uniqueness.

Remark that the dynamical twist is obtained  in  \cite{ABRR} in the form of infinite product
via Picard iterations.
Even in a finite dimensional representation, it involves infinite number of factors,
while the series (\ref{geom F}) turns to a finite sum. The highest degree of $F^{(k)}$  in the truncate series for
$\hat F \in \End(V)\tp \hat U_q(\b_-)$
is equal to the length of longest path in $\La(V)$.

\vspace{20pt}
{\bf Acknowledgements}.
This research is supported in part by the RFBR grant 15-01-03148. 
We are grateful to the Max-Plank Institute for Mathematics in Bonn for hospitality and
excellent research atmosphere.

\end{document}